\documentclass[12pt]{amsart}
\usepackage{amscd,amssymb}
\usepackage[mathscr]{eucal}
\usepackage[english]{babel}
\input xy
\usepackage[all]{xy}

\usepackage{amsmath}
\usepackage{comment}

\usepackage[latin1]{inputenc}%
\usepackage{a4}

\newlength{\itemlaenge}

\newtheoremstyle{mytheorem}
  {}
  {}
  {\slshape}
  {}
  {\scshape}
  {.}
  { }
  {}

\newtheoremstyle{mydefinition}
  {}
  {}
  {\upshape}
  {}
  {\scshape}
  {.}
  { }
  {}

\theoremstyle{mytheorem}
\newtheorem{lemma}{Lemma}[section]
\newtheorem{prop}[lemma]{Proposition}
\newtheorem*{prop*}{Proposition}
\newtheorem{prop_intro}{Proposition}
\newtheorem{cor}[lemma]{Corollary}
\newtheorem{cor_intro}[prop_intro]{Corollary}
\newtheorem{thm}[lemma]{Theorem}
\newtheorem{thm_intro}[prop_intro]{Theorem}

\newtheorem*{thm*}{Theorem}
\theoremstyle{mydefinition}
\newtheorem{rem}[lemma]{Remark}
\newtheorem*{rem*}{Remark}

\newtheorem*{notation*}{Notation}

\newtheorem*{warning*}{Warning}

\newtheorem*{defi*}{Definition}

\numberwithin{equation}{section}

\newcommand{\bqn}{\begin{equation*}}
\newcommand{\eqn}{\end{equation*}}
\newcommand{\bq}{\begin{equation}}
\newcommand{\eq}{\end{equation}}
\newcommand{\ba}{\begin{aligned}}
\newcommand{\ea}{\end{aligned}}
\newcommand{\be}{\begin{enumerate}}
\newcommand{\ee}{\end{enumerate}}

\newcommand{\thismonth}{\ifcase\month 
  \or January\or February\or March\or April\or May\or June%
  \or July\or August\or September\or October\or November%
  \or December\fi}

\newcommand{\PSL}{\operatorname{PSL}}

\newcommand{\FF}{{\mathbb F}}

\newcommand{\QQ}{{\mathbb Q}}
\newcommand{\RR}{{\mathbb R}}

\newcommand{\Ff}{{\mathcal F}}

\newcommand{\Pp}{{\mathcal P}}

\newcommand{\Ss}{{\mathcal S}}
\newcommand{\Tt}{{\mathcal T}}

\newcommand{\Zz}{{\mathcal Z}}

\def\linfo{\ell_{\mathrm{odd}}^\infty}

\newcommand{\vare}{\epsilon}

\newcommand{\G}{\Gamma}

\newcommand{\bb}{{\partial}}

\def\h{{\rm H}}
\def\hb{{\rm H}_{\rm b}}

\def\cb{{\rm C}_{\rm b}}

\def\cbhar{\ell_{\mu,{\mathrm{alt}}}^\infty}
\def\linfa{\ell_{\mathrm{alt}}^\infty}
\def\ch{{\rm C}}

\def\one{\mathbf{1\kern-1.6mm 1}}

\def\id{{\it I\! d}}

\def\h2{{\operatorname{H_2}}}
\def\h1{{\operatorname{H_1}}}

\def\id{{\operatorname{Id}}}

\def\PSL{\operatorname{PSL}}

\def\to{\rightarrow}

\def\hb{{\rm H}_{\rm b}}

\def\h{{\rm H}}

\renewcommand{\phi}{\varphi}

\def\No{N\raise4pt\hbox{\tiny o}\kern+.2em}
\def\no{n\raise4pt\hbox{\tiny o}\kern+.2em}
\def\bsl{\backslash}
\newcommand{\R}{\textup{Hom}}

\newcommand{\Gg}{\mathcal{G}}

\begin{document}

\title[Isometric embeddings in bounded cohomology]{Isometric embeddings in bounded cohomology}

\author[M. Bucher et al.]{M. Bucher}
\address{Section de Math\'ematiques, Universit\'e de Gen\`eve, 
2-4 rue du Li\`evre, Case postale 64, 1211 Gen\`eve 4, Suisse}
\email{Michelle.Bucher-Karlsson@unige.ch}

\author[]{M. Burger}
\address{Department Mathematik, ETH Z\"urich, 
R\"amistrasse 101, CH-8092 Z\"urich, Switzerland}
\email{burger@math.ethz.ch}

\author[]{R. Frigerio}
\address{Dipartimento di Matematica, Universit\`a di Pisa, Largo B. Pontecorvo 5, 56127 Pisa, Italy}
\email{frigerio@dm.unipi.it}

\author[]{A. Iozzi}
\address{Department Mathematik, ETH Z\"urich, 
R\"amistrasse 101, CH-8092 Z\"urich, Switzerland}
\email{iozzi@math.ethz.ch}

\author[]{C. Pagliantini}
\address{Fakult\"at f\"ur Mathematik
Universit\"at Regensburg
Universit\"atsstrasse 31
93053 Regensburg, Germany}
\email{Cristina.Pagliantini@mathematik.uni-regensburg.de}

\author[]{M. B. Pozzetti}
\address{Department Mathematik, ETH Z\"urich, 
R\"amistrasse 101, CH-8092 Z\"urich, Switzerland}
\email{beatrice.pozzetti@math.ethz.ch}

\thanks{Michelle Bucher was supported by Swiss National Science Foundation 
project PP00P2-128309/1.  Alessandra Iozzi was partial supported by the 
Swiss National Science Foundation project 2000021-127016/2.  
Marc Burger, Alessandra Iozzi and Beatrice Pozzetti were partially supported 
by the Swiss National Science Foundation project 200020-144373. The first five
named authors thank the Institute Mittag-Leffler in Djursholm, Sweden, 
for their warm hospitality during  the preparation of this paper. Likewise, Marc Burger 
and Alessandra Iozzi are grateful to the Institute for Advanced Study in Princeton, NJ 
for their support.}



\date{\today}

\begin{abstract} 
This paper is devoted to the construction of norm-preserving maps between bounded cohomology groups. 
For a graph of groups with amenable edge groups 
we construct an isometric embedding of the direct sum of the bounded cohomology of the vertex groups 
in the bounded cohomology of the fundamental group of the graph of groups. 
With a similar technique we prove that if $(X,Y)$ is a pair of CW-complexes 
and the fundamental group of each connected component of $Y$ is amenable,
the isomorphism between the relative bounded cohomology of $(X,Y)$ and 
the bounded cohomology of $X$ in degree at least $2$ is isometric.
As an application we provide easy and self-contained proofs of Gromov's Equivalence Theorem and 
of the additivity of the simplicial volume with respect to gluings along $\pi_1$-injective boundary components with amenable fundamental group.

\end{abstract}
\maketitle
%
%
%
%
\section{Introduction}\label{sec:intro}
Bounded cohomology of groups and spaces was introduced by Gromov in the mid seventies \cite{Gromov_82}
and can be dramatically different from their usual cohomology.
For example, in the context of bounded cohomology,
the lack of a suitable Mayer--Vietoris sequence 
prevents the use of the usual ``cut and paste'' techniques 
exploited in the computation of singular cohomology.
Another peculiarity of
bounded cohomology  is that, in positive degree, the bounded cohomology of any amenable group 
(or of any space with amenable fundamental group) vanishes.

Using the Mayer--Vietoris sequence it is easy to show that, in positive degree, 
the cohomology of a free product of groups is isomorphic to the direct sum of the cohomologies of the factors.
The main result of this paper provides an analogous result in the context of bounded cohomology.  
Since amenable groups are somewhat invisible to bounded cohomology, it is natural to extend the object of our study from
free products to amalgamated products (or HNN extensions) along amenable subgroups. 
In order to treat both these cases at the same time we will exploit notions and 
results coming from the Bass--Serre theory of graphs of groups (we refer the reader to 
\S~\ref{sec:proof:graph} for a brief account on this topic).

\medskip
For every group $\Gamma$ we denote by $\hb^\bullet (\Gamma)$ the bounded cohomology of $\Gamma$ 
with trivial real coefficients, endowed with the $\ell^\infty$-seminorm.
If $\mathcal{G}$ is a graph of groups based on the graph $G$,
we denote by $V(G)$ the set of vertices of $G$, and by $\Gamma_v$, $v\in V(G)$,
the vertex groups of $\mathcal{G}$. Moreover, if $G$ is finite, then for every element
$(\varphi_{1},\ldots,\varphi_{k})\in \oplus_{v \in V(G)} \hb^n (\Gamma_v)$ we set
$$
\| (\varphi_{1},\ldots,\varphi_{k})\|_\infty=\max \{\|\varphi_{1}\|_\infty,\ldots,\|\varphi_{k}\|_\infty\}\ .
$$
We denote by $\Gamma$ the fundamental group of $\mathcal{G}$, by $i_v\colon \Gamma_v\hookrightarrow \Gamma$  the inclusion of $\Gamma_v$ into $\Gamma$, and
by $\h(i_v)\colon \hb^n(\Gamma)\to \hb^n(\Gamma_v)$ the map induced by $i_v$ on bounded cohomology.

The main result of our paper is the following:

\begin{thm_intro}\label{thm:bounded:graph}
Let $\Gamma$ be the fundamental group of a graph of groups $\mathcal{G}$ based on the 
finite graph $G$.
Suppose that every vertex group of $\mathcal{G}$ is countable, and that every edge group
of $\mathcal{G}$ is amenable.
Then for every $n\in\mathbb{N}\setminus\{0\}$ there exists an isometric embedding 
$$
\Theta\colon
\bigoplus_{v\in V(G)} \hb^n (\Gamma_v) \longrightarrow \hb^n(\Gamma)
$$
which provides a right inverse to the map
$$
 \bigoplus_{v\in V(G)} \h(i_v) \ \colon \ \hb^n(\Gamma)\longrightarrow  \bigoplus_{v\in V(G)} \hb^n (\Gamma_v) \ .
$$
\end{thm_intro}

The isometric embedding $\Theta$ is in general far from being an isomorphism: for example,
the real vector spaces $\hb^2(\mathbb{Z}*\mathbb{Z})$ and $\hb^3(\mathbb{Z}*\mathbb{Z})$ 
are infinite-dimensional~\cite{Brooks, Mitsumatsu} (see also \cite{Rolli} for a very beautiful and slick proof), 
while $\hb^n(\mathbb{Z})\oplus \hb^n(\mathbb{Z})=0$ for every $n\geq 1$, since $\mathbb{Z}$ is amenable. 

Moreover the hypothesis that edge groups are amenable is necessary, as the following example shows.
Let $\Gamma<\PSL(2,\QQ_p)\times\PSL(2,\QQ_q)$ be an irreducible torsionfree cocompact lattice.
Then $\Gamma$ acts faithfully and without inversion on the Bruhat--Tits tree $\Tt_{p+1}$
associated to $\PSL(2,\QQ_p)$ with an edge as fundamental domain and 
is therefore the amalgamated product $\FF_a\ast_{\FF_c}\FF_b$ of two non-Abelian free groups
over a common finite index subgroup.  It follows from \cite[Theorem~1.1]{BM1}
that $\hb^2(\Gamma)$ is finite dimensional, while $\hb^2(\FF_a)$ is infinite dimensional.

Our construction of the map $\Theta$ in Theorem~\ref{thm:bounded:graph} relies on 
the analysis of the action of $\Gamma$ on its Bass--Serre tree, which allows us  
to define a projection from combinatorial simplices in $\Gamma$ to simplices with values in the vertex groups. 
Our construction is inspired by \cite[p.~54]{Gromov_82}
and exploits the approach to bounded cohomology developed by
Ivanov~\cite{Ivanov}, Burger and Monod~\cite{BM2, Monod_book}.

Surprisingly enough, the proof of Theorem~\ref{thm:bounded:graph} runs into additional difficulties in the case of degree 2. 
In that case, even to define the map $\Theta$, 
it is necessary to use the fact that bounded cohomology can be computed via the complex of pluriharmonic functions~\cite{BM1}, 
and that such a realization has no coboundaries in degree 2 
due to the double ergodicity of the action of a group on an appropriate Poisson boundary \cite{Ka, BM2}.

\medskip
A simple example of a situation to which Theorem~\ref{thm:bounded:graph} applies is the one 
in which $G$ consists only of one edge $e$ with vertices $v$ and $w$.  In this case,
we can realize $\Gamma_v\ast_{\Gamma_e}\Gamma_w$ as the fundamental group of a space $X$ 
that can be decomposed as $X=X_v\cup X_w$, where $X_v\cap X_w$ has amenable fundamental group.
A fundamental result by Gromov implies that the bounded cohomology of a CW-complex\footnote{See \cite{B} for a more general version for all path connected spaces.} 
is isometrically isomorphic to the bounded cohomology of its fundamental group \cite[p. 49]{Gromov_82}.
Using this, Theorem~\ref{thm:bounded:graph} specializes to the statement that there is an isometric embedding
\bqn
\hb^n(X_v)\oplus\hb^n(X_w)\hookrightarrow\hb^n(X)
\eqn
that is a left inverse to the restriction map.  This forces classes in the image of the map
to have some compatibility condition on $X_v\cap X_w$ and leads naturally to considering 
the bounded cohomology of $X_v$ and $X_w$ relative to $X_v\cap X_w$.
 
To this purpose, let $(X,Y)$ be a pair of countable CW-complexes, 
and denote by $j^n\colon \cb^n(X,Y)\to \cb^n(X)$ the inclusion of relative 
bounded cochains into bounded cochains.

\begin{thm_intro}\label{thm:main}  Let $X\supseteq Y$ be a pair of countable CW-complexes.  
Assume that each connected component of $Y$ has amenable fundamental group.
Then the map
\bqn
\xymatrix@1{
\h ( j^n)\colon \hb^n(X,Y)\ar[r]
&\hb^n(X)
}
\eqn
is an isometric isomorphism for every $n\geq2$.
\end{thm_intro}

The amenability of $\pi_1(Y)$ insures immediately, using the long exact sequence in relative bounded cohomology, 
the isomorphism of $\hb^n(X,Y)$ and $\hb^n(X)$, but the fact that this isomorphism is isometric is, 
to our knowledge, not contained in Gromov's paper and requires a proof.
This result  was obtained independently by Kim and Kuessner~\cite{KK},
using the rather technical theory of multicomplexes.
Our proof of Theorem~\ref{thm:main} uses instead in a crucial way the construction of an amenable 
$\pi_1(X)$-space thought of as a discrete approximation of the pair $(\widetilde X, p^{-1}(Y))$, 
where $p\colon \widetilde{X}\to X$ is a universal covering. 
The same technique is at the basis of the proof of Theorem~\ref{thm:bounded:graph}.

\subsection*{Applications}
In the second part of the paper we show how Theorems~\ref{thm:bounded:graph} and \ref{thm:main} can be used
to provide simple, self-contained proofs of two theorems in bounded cohomology due to Gromov and some new consequences.
The proofs of Gromov's results available in the literature rely on the  theory of multicomplexes \cite{Gromov_82, Kuessner}.

The first of our applications is Gromov's additivity theorem for the simplicial volume,
from which we deduce the behavior of the simplicial volume under generalized Dehn fillings,
thus generalizing a result of Fujiwara and Manning.
We then establish Gromov's equivalence theorem and 
give an $\ell^1$-homology version of Theorem~\ref{thm:main} due to Thurston.

\subsubsection*{Additivity of the Simplicial Volume}\label{subsubsec:simplvol}
The simplicial volume is a homotopy invariant of manifolds introduced by Gromov in his seminal paper~\cite{Gromov_82}. 
If $M$ is a connected, compact and oriented manifold with
(possibly empty) boundary, then the simplicial volume of $M$ is equal to the $\ell^1$-seminorm
of the fundamental class of $M$ (see \S~\ref{sec:additivity} for the precise definition).
It is usually denoted by $\|M\|$ if $M$ is closed, and by $\|M,\bb M\|$
if $\bb M\neq\emptyset$. The simplicial volume may be defined also in the context of open manifolds~\cite{Gromov_82},
but in this paper we will restrict our attention  to compact ones.
More precisely, unless otherwise stated, every manifold will be assumed to be connected, compact and oriented.

The explicit computation of non-vanishing simplicial volume is only known 
for complete finite-volume hyperbolic manifolds (see \cite{Gromov_82, Thurston_notes} for the closed case 
and e.g.~\cite{stefano, FriPag, Fujiwara_Manning, Bucher_Burger_Iozzi_mostow} for the cusped case) 
and  for manifolds locally isometric to the product of two hyperbolic planes \cite{Bucher3} 
(see also \cite{LoehSauer,BKK} for the non-compact case with amenable cusp groups). 
Gromov's Additivity Theorem can be used to establish more computations of the simplicial volume 
by taking connected sums or gluings along $\pi_1$-injective boundary components 
with amenable fundamental group. For example the simplicial volume of a closed $3$-manifold $M$  equals the sum of the simplicial volumes of its hyperbolic pieces~\cite{Soma}.

Furthermore, without aiming at being exhaustive, here we just mention that 
Gromov Additivity Theorem has been also exploited in studying
the possible degrees of maps between manifolds~\cite{Rong, Wang, Derbez2,  Boi,  Derbez3,  Derbez1}, 
in establishing results about the behavior of manifolds under collapse~\cite{Orbi, Bess}, 
and in various other areas of low-dimensional topology~\cite{Adams, Murakami, Boi2, Bel, JNWZ, Storm, Bruno, Koj,  Agol}.

\begin{thm_intro}[Gromov Additivity Theorem]\label{simpl:thm}
Let $M_1,\ldots,M_k$ be $n$-dimensional manifolds, $n\geq 2$, 
suppose that the fundamental group of every boundary component of every $M_j$ is amenable, and let $M$
be the manifold obtained by gluing $M_1,\ldots,M_k$ along (some of) their boundary components.
Then
$$
\| M,\bb M\|\leq \|M_1,\bb M_1\|+\ldots+\|M_k,\bb M_k\|\ .
$$
In addition, if the gluings defining $M$ are compatible, then
$$
\| M,\bb M\|= \|M_1,\bb M_1\|+\ldots+\|M_k,\bb M_k\|\ .
$$
\end{thm_intro}

Here a gluing $f:S_1\to S_2$ of two boundary components $S_i\subseteq \partial M_{j_i}$ is called \emph{compatible} 
if $f_*(K_1)=K_2$ where $K_i$ is the kernel of the map $\pi_1(S_i)\to\pi_1(M_{j_i})$ induced by the inclusion.
\bigskip

An immediate consequence of this theorem is the fact that the simplicial volume is additive with respect to connected sums: given two $n$-dimensional manifolds $M_1$, $M_2$, if $n\geq 3$ and the fundamental group of every boundary component of $M_i$ is amenable, then 
$$
\| M_1\# M_2, \partial (M_1\# M_2)\|=\|M_1,\partial M_1\|+\|M_2,\partial M_2\|,
$$
where $M_1\# M_2$ is constructed by removing an open ball from the interior of $M_i$ and gluing the obtained manifolds along the boundary spheres.

\subsubsection*{Generalized Dehn Fillings}
A consequence of the first part of Theorem~\ref{simpl:thm}
is an easy proof of a result of Fujiwara and Manning~\cite{Fujiwara_Manning} about generalized Dehn fillings. 
Let $n\geq 3$ and let $M$ be a compact orientable
$n$-manifold such that 
$\partial M=N_1\cup\ldots\cup N_m$, where $N_i$ is an $(n-1)$-torus for every $i$.
For each $i\in\{1,\dots, m\}$ we put on $N_i$ a flat structure,
and we choose a totally geodesic $k_i$-dimensional
torus  $T_i\subseteq N_i$, where $1\leq k_i\leq n-2$.
Each $N_i$ is foliated by parallel copies of $T_i$ with leaf space
$L_i$ homeomorphic to a $(n-1-k_i)$-dimensional torus. 
The \emph{generalized Dehn filling} $R=M(T_1,\dots,T_m)$ is defined 
as the quotient of ${M}$ obtained  by collapsing $N_i$ on $L_i$ for every $i\in\{1,\dots,m\}$. 
Observe that unless $k_i=1$ for every $i$, the quotient $R$ is not a manifold. In any case, 
being a pseudomanifold, $R$ always admits a fundamental class,
whence a well-defined simplicial volume. Fujiwara and Manning proved that, if 
the interior of $M$ admits a complete finite-volume hyperbolic structure, then
the inequality $\|R\|\leq \|M,\partial M\|$ holds. 
Their argument
easily extends to the case in which the
fundamental group of $M$ is residually finite and the inclusion of each boundary
torus in $M$ induces an injective map on fundamental groups. As a consequence
of Theorem~\ref{simpl:thm} we obtain the following generalization of Fujiwara and Manning's
result:

\begin{cor_intro}\label{Dehn filling}
Let $M$ be a compact orientable
$n$-manifold with boundary given by a union of tori, and let $R$ 
be a generalized Dehn filling of $M$.
Then
$$\|R \|\leq \|M,\partial M\|.$$
\end{cor_intro}

\subsubsection*{Equivalence of Gromov Norms}
In \cite{Gromov_82} Gromov introduced a one-parameter family of seminorms on $\h_n(X,Y)$.
More precisely, let $\theta\in [0,\infty)$ and consider the norm $\|\cdot\|_1(\theta)$ on $C_n(X)$ 
defined by $\|c\|_1(\theta)=\|c\|_1+\theta\|\partial_n c\|_1$. 
Every such norm is equivalent to the usual norm $\|\cdot\|_1=\|\cdot\|_1(0)$ for 
every $\theta\in[0,\infty)$ and induces a quotient seminorm on relative homology, 
still denoted by $\|\cdot \|_1(\theta)$.
By passing to the limit, one can also define a seminorm $\|\cdot\|_1(\infty)$ that, 
however, may be non-equivalent to $\|\cdot\|_1$.
The following result is stated by Gromov in~\cite{Gromov_82}.

\begin{thm_intro}[{Equivalence Theorem, \cite[p.~57]{Gromov_82}}]\label{thm:equi}
Let $X\supseteq Y$ be a pair of countable CW-complexes, and let $n\geq 2$.
If the fundamental groups of all connected components of $Y$ are amenable,
then the seminorms $\|\cdot\|_1(\theta)$ on $\h_n(X,Y)$
coincide for every $\theta\in[0,\infty]$.
\end{thm_intro} 

In order to prove Theorem~\ref{thm:equi}, we establish two isometric isomorphisms of independent interest 
(see Lemma \ref{lem: theta iso} and Proposition  \ref{fund:prop}), 
using the homological construction of a mapping cone complex and 
considering a one-parameter family of seminorms in bounded cohomology introduced by Park~\cite{Park}.

As noticed by Gromov, Theorem~\ref{thm:equi} admits the following equivalent formulation, 
which is inspired by Thurston~\cite[\S 6.5]{Thurston_notes} and 
plays an important role in several results about the (relative) simplicial volumes of gluings and fillings:

\begin{cor_intro}\label{Thurston's version}
Let $X\supseteq Y$ be a pair of countable CW-complexes, and
suppose that the fundamental groups of all the components of $Y$ are amenable.
Let $\alpha\in \h_n(X,Y)$, $n\geq 2$. Then,
for every $\epsilon>0$, there exists an element $c\in \ch_n(X)$ with $\partial_n c\in \ch_{n-1}(Y)$ 
such that $[c]=\alpha\in \h_n(X,Y)$, 
$\|c\|_1<\|\alpha\|_1+\epsilon$ and $\|\partial_n c\|_1<\epsilon$.
\end{cor_intro}
\begin{proof}
 Let $\theta=(\|\alpha\|_1+\epsilon)/\epsilon$. By Theorem~\ref{thm:equi} 
 we know that $\|\cdot\|_1(\theta)$ induces the norm $\|\cdot\|_1$ in homology, 
 so we can find a representative $c \in  \ch_n(X)$ of $\alpha$ 
 with $\|c\|_1(\theta)=\|c\|_1+\theta\|\partial_n c\|_1\leq \|\alpha\|_1+\epsilon$.
 This implies that $\|c\|_1\leq\|\alpha\|_1+\epsilon$ and $\|\partial_n c\|_1\leq (\|\alpha\|_1+\epsilon)/\theta = \epsilon$.
\end{proof}

\section{Resolutions in bounded cohomology}\label{sec: Resolutions}

Let $X$ be a space, where here and in the sequel by a space we will always mean a countable CW-complex. 
We denote by $\cb^\bullet(X)$ the complex of bounded real valued
singular cochains on $X$ and, if $Y\subset X$ is a subspace, 
by $\cb^\bullet(X,Y)$ the subcomplex of those bounded cochains
that vanish on simplices with image contained in $Y$.  All these spaces
are endowed with the $\ell^\infty$-norm and the corresponding cohomology groups
are equipped with the corresponding quotient seminorm.

For our purposes, it is important to  observe that the universal covering map
$p:\widetilde X\to X$ induces an isometric identification of the complex 
$\cb^\bullet(X)$ with the complex $\cb^\bullet(\widetilde X)^\Gamma$
of $\Gamma:=\pi_1(X)$-invariant bounded cochains on $\widetilde X$.
Similarly, if $Y':=p^{-1}(Y)$, we obtain an isometric identification
of the complex $\cb^\bullet(X,Y)$ with the complex 
$\cb^\bullet(\widetilde X,Y')^\Gamma$ of $\Gamma$-invariants
of $\cb^\bullet(\widetilde X,Y')$.  

The main ingredient in the proof of Theorem~\ref{thm:main}
is the result of Ivanov \cite{Ivanov} that the complex of  $\Gamma$-invariants of \bqn
\xymatrix{
0\ar[r]
&\cb^0(\widetilde X)\ar[r]
&\cb^1(\widetilde X)\ar[r]
&\dots
}
\eqn
computes the bounded cohomology of $\Gamma$. In fact, we will use the more precise statement 
that the obvious augmentation of the complex above is a strong resolution of $\RR$ by relatively injective Banach $\Gamma$-modules 
(see \cite{Ivanov} for the definitions of strong resolution and relatively injective module). 

By standard homological algebra techniques  \cite{Ivanov}, 
it follows from the fact that $\cb^n(\widetilde X)$ is a strong resolution by $\Gamma$-modules and 
$\ell^\infty(\Gamma^{\bullet+1})$ is a cochain complex  
of relatively injective $\Gamma$-modules that there exists   a $\Gamma$-morphism of complexes
\bq \tag{$\Diamond
$}\label{eq: g_n}
\xymatrix@1{
g^n:\cb^n(\widetilde X)\ar[r]
&\ell^\infty(\Gamma^{n+1})
}
\eq
extending the identity of $\R$, and such that $g^n$ is norm non-increasing, i.e. $\|g^n\|\leq1$, for $n\geq0$. 
This map induces the isometric isomorphism $\hb^\bullet(X)\rightarrow \hb^\bullet(\Gamma)$, and will be referred
to as Ivanov's map.

\bigskip

The second result we need lies at the basis of the fact that the bounded cohomology of $\Gamma$ 
can be computed isometrically from the complex of bounded functions on any amenable $\Gamma$-space. The notion of 
amenable space 
was introduced by Zimmer~\cite{Zimmer} in the context of
actions of topological groups
on standard measure spaces
(see e.g.~\cite[\S 5.3]{Monod_book} for several equivalent definitions).
In our case of interest, i.e.~when $\G$ is a discrete countable group acting on a countable set $S$ 
(which may be thought as endowed with the discrete topology), the amenability of $S$ as a $\G$-space amounts to the amenability
of the stabilizers in $\G$ of elements of $S$~\cite[Theorem 5.1]{AEG}.

\begin{prop}\label{prop:monod}
Let $S$ be an amenable $\G$-set, where $\Gamma$ is a discrete countable group.
Then:
\begin{enumerate}
 \item 
There exists a $\Gamma$-morphism of complexes
\bqn
\xymatrix{
\mu^\bullet\colon\ell^\infty(\Gamma^{\bullet +1})\ar[r]
&\ell_{\mathrm{alt}}^\infty(S^{\bullet +1})\, }
\eqn
extending $\id_\RR:\RR\to \RR$ that is norm non-increasing in every degree.
\item
The cohomology of the complex
$$
\xymatrix{
0\ar[r] & \ell^\infty_{\mathrm{alt}}(S)^\G\ar[r]  &\ell^\infty_{\mathrm{alt}}(S^2)^\G\ar[r]&\ell^\infty_{\mathrm{alt}}(S^3)^\G\ar[r]&\ldots
}
$$ 
is canonically isometrically isomorphic to $\hb^\bullet(\G)$.
\end{enumerate}
\end{prop}
\begin{proof}
 Point (1) is proved in~\cite[Lemma 7.5.6]{Monod_book} (applied to the case $T=\G$), point (2) in~\cite[Theorem 7.5.3]{Monod_book}.
\end{proof}

We point out that the computation of bounded cohomology via alternating cochains
on amenable spaces is natural in the following sense:

\begin{lemma}\label{lemma:funct}
 Let $i\colon \G_1 \to \G$ be an inclusion of countable groups, let $S_1$ be a discrete amenable $\G_1$-space, 
 and $S$ a discrete amenable $\G$-space. If $\phi\colon S_1\to S$ is equivariant with respect to $i$, then the following diagram commutes:
 \bqn\xymatrix{\Zz\linfa(S^{\bullet+1})^\G\ar[r]\ar[d]^{\phi^*}&\hb^\bullet(\G)\ar[d]^{i^*}\\ \Zz\linfa(S_1^{\bullet+1})^{\G_1}\ar[r]&\hb^\bullet(\Gamma_1).}\eqn
\end{lemma}

\bigskip

The third and last ingredient we need is a result from \cite{BM1} 
where it is shown that the bounded cohomology of $\Gamma$ is realized by yet another complex, 
namely the resolution via $\mu$-pluriharmonic functions.

Let $\mu$ be a symmetric probability measure on $\Gamma$ and 
denote by $\ell_{\mu, \mathrm{alt}}^\infty(\Gamma^{n+1})$ the subcomplex of $\ell_{\mathrm{alt}}^\infty(\Gamma^{n+1})$ 
consisting of $\mu$-pluriharmonic functions on $\Gamma^{n+1}$, i.e.~of elements $f\in\ell^\infty_\mathrm{alt} (\G^{n+1})$
such that 
$$
f(g_0,\ldots,g_n)=\int_{\G^{n+1}}f(g_0\gamma_0^{-1},\ldots,g_n\gamma_n^{-1})d\mu(\gamma_0)\ldots d\mu(\gamma_n)
$$
for every $(g_0,\ldots,g_n)\in\G^{n+1}$.
By~\cite[Lemma 3.13]{BM1}, the inclusion $\cbhar(\Gamma^\bullet)\hookrightarrow \ell_{\mathrm{alt}}^\infty(\Gamma^\bullet)$ induces isometric
isomorphisms in cohomology. 

Moreover, if $(B,\nu)$ is the Poisson boundary of $(\Gamma,\mu)$, it is proven in \cite[Proposition 3.11]{BM1} that the Poisson transform
\begin{displaymath}
 \begin{array}{c}
 \Pp:L^\infty_{\mathrm{alt}}(B^{n+1},\nu^{\otimes n+1})\to\cbhar(\Gamma^{n+1})\\
 \Pp(f)(g_0,\ldots,g_n)=\int_{B^{n+1}}f(g_0\xi_0,\ldots,g_n\xi_n)d\nu(\xi_0)\ldots d\nu(\xi_n)
 \end{array}
\end{displaymath}
is a $\Gamma$-equivariant isometric isomorphism.

The main theorem of~\cite{Ka} (see also~\cite[Theorem 0.2]{BM2} and \cite[Proposition 4.2]{BI} for the case of finitely generated groups)
implies that, if the support of $\mu$ generates $\Gamma$, then the action of $\G$ on $B$ is doubly ergodic, in particular
$$\cbhar(\G^2)^\G=L^\infty_{\mathrm{alt}} (B^2,\nu^{\otimes 2})^\G=0\ ,$$
and the 
projection of ${\mathcal Z}(\ell^\infty_{\mathrm{alt}}(\G^3)^\G)$ onto $\hb^2(\G)$ restricts to an isometric isomorphism
between the space  of $\G$-invariant
$\mu$-pluriharmonic  alternating cocycles ${\mathcal Z}\cbhar(\G^3)^\G$ and the second bounded cohomology module of $\G$. This implies

\begin{prop}\label{prop:har}
 Let $\G$ be a countable group and $\mu$ a symmetric probability measure whose support generates $\G$. Then there is an isometric linear section
 $$
 \sigma\colon \hb^2(\G)\to\Zz\ell^\infty(\G^3)^\G.
 $$
 of the projection defining bounded cohomology.
\end{prop}

\section{Relative Bounded Cohomology: Proof of Theorem~\ref{thm:main}}\label{sec:proof:main}
Let $(X,Y)$ be a pair of countable CW-spaces. Assume that $X$ is
connected and the fundamental group of every component of $Y$ is amenable. Let  $p\colon \widetilde X\to X$ be the universal covering map, set
$\Gamma:=\pi_1(X)$ and let $Y=\sqcup_{i\in I}C_i$ the decomposition of $Y$
into the union of its connected components. If  $\check C_i$ is a choice of a connected component of $p^{-1}(C_i)$ 
and $\Gamma_i$ denotes the stabilizer of $\check C_i$ in $\Gamma$, then
\bqn
p^{-1}(C_i)=\bigsqcup_{\gamma\in\Gamma/\Gamma_i}\gamma\check C_i\,.
\eqn

\noindent
The group $\G$ acts by left translations on the set
\bqn
S:=\Gamma\sqcup\bigsqcup_{i\in I}\Gamma/\Gamma_i\ .
\eqn

\noindent Being a quotient of $\pi_1(C_i)$, the group $\Gamma_i$ is amenable,
so $S$ is an amenable $\G$-space.
We  define a $\G$-equivariant measurable retraction $r\colon\widetilde X\to S$ as follows: 
let $\Ff\subset\widetilde X\smallsetminus Y'$ be a fundamental domain 
for the $\Gamma$-action on $\widetilde X\smallsetminus Y'$, where $Y'=p^{-1}(Y)$.
Define the map $r$ as follows:
\bqn
r(\gamma x):=
\begin{cases}
\hphantom{\Gamma}\,\gamma\in\Gamma&\text{ if }x\in\Ff,\\
\gamma\Gamma_i\in \Gamma/\Gamma_i &\text{ if } x \in\check C_i\,.
\end{cases}
\eqn
For every $n\geq0$ define 
\bqn
\xymatrix@1{
r^n\colon \ell^\infty_{\mathrm{alt}}(S^{n+1})\ar[r]&\cb^n(\widetilde X)}
\eqn
by 
\bqn
r^n(c)(\sigma)=c(r(\sigma_0),\dots,r(\sigma_n))\,,
\eqn
where $c\in\ell^\infty_{\mathrm{alt}}(S^{n+1})$ and $\sigma_0,...,\sigma_n\in \widetilde{X}$ 
are the vertices of a singular simplex $\sigma\colon \Delta^n\to\widetilde X$.
Clearly $(r^n)_{n\geq0}$ is a $\Gamma$-morphism of complexes extending the identity of $\RR$
and $\|r^n\|\leq1$ for all $n\geq0$.

Observe that if $n\geq1$ and $\sigma(\Delta^n)\subset Y'$, 
then there are $i\in I$ and $\gamma\in\Gamma$ such that $\sigma(\Delta^n)\subset\gamma\check C_i$.
Thus
\bqn
r(\sigma_0)=\dots=r(\sigma_n)=\gamma\Gamma_i
\eqn
and
\bqn
r^n(c)(\sigma)=c(\gamma\Gamma_i,\dots,\gamma\Gamma_i)=0\,,
\eqn
since $c$ is alternating.
This implies that 
the image of $r^n$ is in $\cb^n(\widetilde X,Y')$. Thus we can write $r^n=j^n\circ r_1^n$,
where $j^n:\cb^n(\widetilde X,Y')\hookrightarrow\cb^n(\widetilde X)$ is the inclusion
and $r_1^n:\ell^\infty_{\mathrm{alt}}(S^{n+1})\to\cb^n(\widetilde X,Y')$ is a norm non-increasing $\Gamma$-morphism
that induces a norm non-increasing map
in cohomology 
\bqn
\xymatrix@1{\h(r_1^n):\h^n(\ell^\infty_{\mathrm{alt}}(S^{\bullet+1})^\Gamma)\ar[r]&\hb^n(X,Y)}\,,
\eqn
for $n\geq1$.

Using the map $g^n$ defined in (\ref{eq: g_n}) and the map $\mu^n$ provided by 
Proposition~\ref{prop:monod},
we have the following diagram
\begin{equation*}
\xymatrix{\cb^n(\widetilde{X})  \ar@{-->}[rrrd]_{\mathrm{extends \ Id}_\mathbb{R}} \ar[r]^{g^n} 
& \ell^\infty(\Gamma^{n+1})  \ar[r]^{\mu^n} & \ell^\infty_{\mathrm{alt}}(S^{n+1}) \ar[rd]^{r^n}\ar[r]^{r_1^n}_{\mathrm{for \ }n\geq 1} 
&\cb^n(\widetilde{X},Y') \ar[d]^{j^n} \\
&&&\cb^n(\widetilde{X}),} 
\end{equation*}
where the dotted map is the composition $r^n\circ \mu^n \circ g^n$ 
which is a $\Gamma$-morphism  of strong resolutions by relatively injective modules
extending the identity, and hence induces the identity on 
$\hb^n(X)=\h^n(\cb^\bullet(\widetilde X)^\Gamma)$.

We proceed now to show that, for $n\geq 2$, the map
\bqn
\xymatrix@1{
\h(j^n):\hb^n(X,Y)\ar[r]&\hb^n(X)
}
\eqn
induced by $j^n$  is an isometric isomorphism in cohomology. 
In view of the long exact sequence for pairs in bounded cohomology and 
the fact that $\hb^\bullet(Y)=0$ in positive degree, we already know that $\h(j^n)$ is an  isomorphism. 
Let us set $\psi^n=r_1^n\circ \mu^n\circ g^n$. 
From the above we have 
\bqn
\h(j^n)\circ \h(\psi^n)=\id_{\hb^n(X)}\, .
\eqn
The conclusion follows form the fact that 
the maps $\h(j^n)$ and $\h(\psi^n)$ are norm non-increasing.

\section{Graphs of Groups: Proof of Theorem~\ref{thm:bounded:graph}}\label{sec:proof:graph}
In order to fix the notation, we recall some 
definitions concerning
graphs of groups, closely following~\cite{Serre}.
A graph $G$ is a pair $(V(G),E(G))$ together with a map $E(G)\to V(G)^2$, 
$e\mapsto(o(e),t(e))$ and a fixed point free involution $e\mapsto\bar e$ of $E(G)$ satisfying $o(e)=t(\bar e)$. 
The set $\bar E(G)$ of geometric edges of $G$ is defined by setting $\bar E(G)=\{\{e,\bar e\}|\,e\in E(G)\}$. 
The geometric realization $|G|$ of a graph 
$G$ is the 1-dimensional CW-complex with one vertex for every element in $V(G)$ and 
one edge for every geometric edge. Its first baricentric subdivision $G'$ has as vertices the set $V(G')=V(G)\sqcup \bar E(G)$.

Let $\mathcal G$ be a graph of groups based on the graph $G$. 
Recall that to every vertex $v\in V(G)$ is associated a group $\G_v$ and 
to every edge $e\in E(G)$ is associated a group $\G_e$ together with an injective homomorphism $h_e\colon \G_e\to \G_{t(e)}$. 
Moreover, it is required that $\G_e=\G_{\bar e}$.  
Let $\G=\pi_1(\mathcal G)$ denote the fundamental group of $\mathcal G$.
By the universal property of the fundamental group of a graph of groups \cite[Corollary 1, p.~45]{Serre}, 
for every $v\in V(G)$, $e\in E(G)$, there exist inclusions $\G_v\to \G$ and $\G_e\to \G$. 
Henceforth we will regard each $\Gamma_v$ and each $\Gamma_e$
just as a subgroup of $\G$. Observe that, since $\Gamma_e=\Gamma_{\bar e}$, it makes
sense to speak about the subgroup $\Gamma_e$ also for $e\in {\bar E}(G)$.

A fundamental result in Bass--Serre theory \cite[Theorem 12, p.~52]{Serre}
implies that $\G$ acts simplicially 
on a tree $T=(V(T), E(T))$, where
$$
V(T)=\bigsqcup_{v\in V(G)}\G/\G_v\, ,\qquad 
 E(T)=\bigsqcup_{e\in {\bar E}(G)}\G/\G_e\ .
$$
The action of
$\G$ on $V(T)$ and $E(T)$ is by left multiplication.
The tree $T$ is known as the \emph{Bass--Serre tree} of $\mathcal{G}$ (or of $\G$, when
the presentation of $\G$ as the fundamental group of a graph of group is understood).
There is an obvious projection $V(T)\to V(G)$ which sends the whole of $\G/\G_v$ to $v$.
This projection admits a preferred section that takes
any vertex $v\in V(G)$ to the coset $1\cdot \G_v\in \G/\G_v$. This allows us to 
canonically identify $V(G)$ with a subset of $V(T)$.

Now we consider the space 
$$ S_{\mathcal G}=\left(\G\times V(G)\right)\ 
\sqcup\bigsqcup_{e\in \bar E(G)}\G/\G_e\ .
$$
We may define an action of $\G$ on $S_\mathcal{G}$ by setting
$g_0\cdot (g,v)=(g_0g,v)$ for every $(g,v)\in \G\times V(G)$ and
$g_0\cdot (g\Gamma_e)=(g_0g)\Gamma_e$ for every $g\Gamma_e\in \G/\G_e$, $e\in {\bar E}(G)$.
 
There exists a $\Gamma$-equivariant projection $p\colon S_{\mathcal G}\to V(T')$ defined as follows: 
$p(g,v)=g\Gamma_v$ for $(g,v)\in \G\times V(G)$, and $p$ is the identity on each
$\G/\G_e$, $e\in {\bar E}(G)$.

Let us now suppose that our graph of groups $\mathcal{G}$ satisfies the hypothesis of Theorem~\ref{thm:bounded:graph}, 
i.e.~every $\Gamma_v$ is countable and every
$\Gamma_e$ is amenable. Under this assumption, both $\G$ and $S_\mathcal{G}$
are countable, and $\G$ acts on $S_\mathcal{G}$ with amenable stabilizers.
As a consequence of Proposition \ref{prop:monod}, 
the bounded cohomology of $\G$ can be isometrically computed from the complex $\ell^\infty_{\mathrm{alt}}(S_{\mathcal G}^{\bullet+1})$.

For every vertex $v\in V(G)$, let $S_v$ be the set 
$$
S_{v}=\G _v\sqcup\bigsqcup_{t(e)=v}\G_v/\G_e,
$$
where we identify $\G_e$ with a subgroup of $\G_v$ via the map $h_e$.
We have an obvious action of $\G_v$ on $S_v$ by left multiplication.
Since every $\Gamma_e$ is amenable, this action turns $S_v$ into an amenable
$\Gamma_v$-space.
 
The inclusion $\phi_v\colon S_{v}\to S_{\mathcal G}$ defined by $\phi_v(g)=(g,v)$
and $\phi_v(g\Gamma_e)=g\Gamma_e$, induces a chain map 
$$
\phi_v^\bullet\colon \ell^\infty_{\mathrm{alt}}(S_{\mathcal G}^{\bullet+1})\to\ell^\infty_{\mathrm{alt}}(S_{v}^{\bullet+1})\ .
$$
By construction, $\phi_v^\bullet$ is equivariant with respect to the inclusion $\G_v\to \G$, 
so Lemma~\ref{lemma:funct} implies that $\phi_v^\bullet$ induces the restriction map in bounded cohomology. 

The following result establishes the existence of a partial retraction
of the chain map $\phi^\bullet=\oplus_{v\in V(G)} \phi_v^\bullet$, and plays a fundamental role in the proof of Theorem \ref{thm:bounded:graph}.

\begin{thm}\label{thm:baricenter}
 There is a (partial) norm non-increasing chain map
$$
\psi^n\colon \bigoplus_{v\in V(G)} \linfa(S_v^{n+1})^{\G_v} \to
\linfa(S_{\Gg}^{n+1})^{\G}\, , \quad n\geq 2\,  
$$
such that the composition $\phi^n\circ\psi^n$ is the identity of $\oplus_{v\in V(G)} \linfa(S_v^{n+1})^{\G_v}$ for every $n\geq 2$.
\end{thm}
\begin{proof}
To define the map $\psi^n$ we need the notion of a \emph{barycenter} of an $(n+1)$-tuple $(y_0,\ldots,y_n)$ in $V(T')^{n+1}$. 
Given a vertex $v\in V( T')$, let $N(v)\subseteq V(T')$ be the set of vertices having combinatorial distance (in $T'$)
at most one from $v$.
The vertex $\bar y\in V(T)\subseteq V(T')$ is  a barycenter of $(y_0,\ldots,y_n)\in V(T')^{n+1}$ if for any $y_i$, $y_j$ in $V(T')\setminus 
\{\bar{y}\}$, $i\neq j$,
the points $y_i$ and $y_j$ belong to different connected components of $|T'|\bsl \{\bar y\}$.
It follows readily from the definitions that there exists at most one barycenter for any $n$-tuple provided that $n\geq 3$.

Let $p\colon S_\mathcal{G}\to V(T')$ be the projection defined above. 
For $v\in V(G)$, let us
identify $S_v$ with $\phi_v(S_v)\subseteq S_{\mathcal{G}}$, and recall that $V(G)$
is canonically identified with a subset
 of $V(T)\subseteq V(T')$. Under these identifications we have 
 $S_v=p^{-1}(N(v))$ for every $v\in V(G)$, and we coherently set 
$S_w=p^{-1}(N(w))\subseteq S_{\mathcal{G}}$ for  every 
$w\in V(T)$.
 
Let us fix $w\in V(T)$. We define a retraction $r^0_w\colon S_{\mathcal{G}}\to S_w$ as follows: 
if $x_0\in S_w$, then $r^0_w(x_0)=x_0$; otherwise, if $y_0$ is the endpoint of the first edge of the combinatorial path $[w,p(x_0)]$ in $T'$, then
 $r^0_w(x_0)$ is the unique preimage of $y_0$ via $p$.
We extend $r^0_w$ to a chain map $r^\bullet_w\colon  S_{\mathcal{G}}^{\bullet+1}\to S_w^{\bullet+1}$  by setting  $r^n_w(x)=(r_w^0(x_0),\ldots,r_w^0(x_n))$
for $x=(x_0,\ldots,x_n)$.
Notice that if $w$ is not a barycenter of $(p(x_0),\ldots,p(x_n))$, 
then the $(n+1)$-tuple $r_w^n(x)$ has at least two coordinates that 
are equal, so any alternating cochain vanishes on $r_w^n(x)$.

We are now ready to define the (partial) chain map $\psi^\bullet$.
Recall that every vertex $w\in V(T)$ is a coset in $\G/\G_v$ for
some $v\in V(G)$. For every $w\in V(T)$ we choose a representative 
$\sigma(w)\in \G$ of $w$, and we observe that $\sigma(w)^{-1}w\in V(G)\subseteq V(T)$.
Let $x\in S_{\mathcal{G}}^{n+1}$, $n\geq 2$.
We have
$\sigma(w)^{-1}r_w^n(x)\in S^{n+1}_{\sigma(w)^{-1}w}$, so
for every $(\oplus_{v\in V(G)} f_v) \in \oplus_{v\in V(G)}\ell^{\infty}_\mathrm{alt} (S^{n+1}_v)$ it makes sense to set
$$
\psi^n( \oplus_{v\in V(G)} f_v) (x)=
\sum_{w\in V(T)}f_{\sigma(w)^{-1}w}(\sigma(w)^{-1}r^n_w(x))\ .
$$

\noindent
Since the $f_v$ are alternating there is at most one non-zero term in the sum, corresponding to the barycenter (if any) 
of $(p(x_0),\ldots,p(x_n))$. Moreover $\psi^n$, $n\geq 2$, is a (partial) chain map and it is easy to check 
that $\psi^n(\oplus_{v\in V(G)} f_v)$ is $\G$-invariant provided 
that $f_v$ is $\G_v$-invariant for every $v\in V(G)$.
\end{proof}

We are now ready to finish the proof of Theorem~\ref{thm:bounded:graph}. 
\begin{proof}[Proof of Theorem \ref{thm:bounded:graph}]
Since the first bounded cohomology of any group vanishes in degree one, 
it is sufficient to consider the case $n\geq2$. Being a norm non-increasing chain map 
defined for every degree $n\geq 2$, $\psi^n$ induces a norm non-increasing map 
$\Theta^n=\h (\psi^n)$ in bounded cohomology for every $n\geq 3$. 
Moreover, being induced by a right inverse of $\phi^n$, the map $\Theta^n$ is a right
inverse of $\h (i_v^n)$ for 
every $n\geq 3$. This implies that $\Theta^n$ is an isometric embedding.

If $n=2$, it is not clear why $\psi^2$ should send coboundaries of bounded 1-cochains to 
coboundaries of bounded 1-cochains. In fact, we will show in the last part of this section 
that this is not the case in general.
This difficulty may be circumvented by
exploiting the fact, proved in \S~\ref{sec: Resolutions}, that any element in $\hb^2(\G)$
admits a special norm-minimizing representative.

In fact let us define the map $\Theta^2$ as the composition of the maps
\bqn
\xymatrix{\oplus \hb^2(\G_v)\ar[r]^-{\oplus\sigma_v}&\oplus\Zz\cbhar(\G_v^3)^{\G_v}\ar[r]^{\oplus \mu_v}&\oplus\Zz\linfa(S_v^3)^{\G_v}\ar[r]^-{\psi^2}&\Zz\linfa(S_{\Gg}^3)^\G\ar[r]&\hb^2(\G)}
\eqn

\noindent where $\sigma_v:\hb^2(\G_v)\to \Zz\cbhar(\G_v^3)^{\G_v}$ is the map described in Proposition~\ref{prop:har}, 
$\mu_v 
$ is the morphism constructed in Proposition~\ref{prop:monod}, and $\psi^2$ is the map of Theorem~\ref{thm:baricenter}.

All the maps involved are norm non-increasing, hence the same holds for $\Theta^2$. 
Moreover, $\Theta^2$ induces a right inverse of the restriction since the following diagram is commutative
$$
\xymatrix{
 & \Zz\linfa(S_\G^3)^\G\ar[r]\ar[d]^{\phi^2} 
& \hb^2(\G)\ar[d] \\ 
\bigoplus \Zz\cbhar(\G_v^3)^{\G_v}\ar[r]^{\oplus\mu_v}
\ar[ru]^{\psi^2\circ\oplus\mu_v}&
\bigoplus \Zz \linfa(S_v^3)^{\G_v}\ar[r] &
\bigoplus \hb^2(\G_v). \ar@/^2pc/[ll]^{\oplus \sigma_v} 
}
$$
This finishes the proof of the theorem.

\end{proof}

\begin{rem}\label{split_quasimorphisms}
Let us now briefly comment on the fact that the map $\psi^2$ does not send, in general, 
coboundaries of bounded 1-cochains to coboundaries of bounded 1-cochains. 
We will be only considering free products, that is the case in which the graph $G$ is a tree and 
all edge groups are trivial. In \cite[Proposition 4.2]{Rolli} Rolli constructed a linear map
\begin{equation}\label{rolli:eq}
\bigoplus_{v\in V(G)}\linfo(\G_v)\to \hb^2(\G)
\end{equation}
and showed that this map is injective. Here $\linfo(\G_v)$ is the set of bounded functions on $\G_v$ such that $f(g^{-1})=-f(g)$.

We denote by $(\bar {\rm C}^\bullet(\G),\bar{d}^\bullet)$ 
(resp.~$(\bar {\rm C}^\bullet_{\rm b}(\G),\bar{d}^\bullet)$) 
the space of inhomogeneous (resp. bounded inhomogeneous) cochains on $\G$, 
and recall that $\bar {\rm C}^\bullet(\G)$ (resp.~$\bar {\rm C}^\bullet_{\rm b}(\G)$) 
is isometrically isomorphic to the corresponding
module of homogeneous $\G$-invariant cochains via the chain map $h^\bullet$ given by 
$h^n(f)(x_0,\ldots,x_n)= f(x_0^{-1}x_1,\ldots, x_{n-1}^{-1}x_n)$. We denote by $\bar {\rm C}^n_{\mathrm{alt}}(\G)$
(resp.~$\bar {\rm C}^n_{{\rm b},\mathrm{alt}}(\G)$)
the subspace of $\bar {\rm C}^n(\G)$ (resp.~$\bar {\rm C}^n_{\rm b}(\G)$) corresponding via $h^n$ to alternating cochains
on $\G^{n+1}$. 

Let $\alpha\colon\oplus\linfo(\G_v)\to \bar{\rm C}_{\mathrm{alt}}^1(\G)$ be defined by 
$\alpha (\oplus f_v)(x)=\sum f_{v_i}(x_i)$, where $x_0\ldots x_n$ is the reduced expression for $x$ and $x_i\in \G_{v_i}$.
Even if the image of $\alpha$ is  not contained in $ \bar {\rm C}^1_{\mathrm{b},\mathrm{alt}}(\G)$ in general,
it is proved in~\cite{Rolli} that the image of the composition $R=\bar{d}\circ \alpha$ consists of bounded cocycles. 
Moreover, $R$ admits the explicit expression 
\begin{equation}\label{expression}
R(\oplus f_v)(x,y)=f_v(\gamma_2)-f_v(\gamma_1\gamma_2)+f_v(\gamma_1)\ ,
\end{equation}
where $a\gamma_1b$ and $b^{-1}\gamma_2c$ are reduced expressions for $x$ and $y$ 
with $\gamma_1$ and $\gamma_2$ maximal subwords belonging 
to the same vertex group $\Gamma_v$ and $\gamma_1\neq \gamma_2^{-1}$. 

Let us now consider the following diagram:
\bqn
\xymatrix{
\oplus\linfo(\G_v)\ar[r]^R\ar[d]^{\mu\circ h^2\circ \bar d}& \Zz\bar {\rm C}^2_{\rm b}(\G)\ar[r]\ar[d]^{\mu\circ h^2}&\hb^2(\G)\ar[d]
\\
\oplus\Zz\linfa(S_v^3)^{\G_v}\ar[r]^{\psi^2}&\Zz\linfa(S_{\mathcal{G}}^3)^{\G}\ar[r]&\hb^2(\G).
}\eqn
Rolli's map~\eqref{rolli:eq} is defined as the composition of the 
horizontal arrows on the top.
We claim that the diagram is commutative.
Since we are in the case of a free product,
we have an obvious identification between $S_{\mathcal{G}}$ and $\G\times V(G')$,
and the map $\mu^\bullet \colon \cb^\bullet(\G)\to \cb^\bullet(S_{\mathcal{G}})$ is
induced by the projection $ \G\times V(T')\to \G$.
The commutativity of the square on the right is now a consequence of Lemma~\ref{lemma:funct}.
To show that the left square commutes,
let us consider a triple $((x_0,v_0),(x_1,v_1),(x_2,v_2))\in S_{\mathcal{G}}^3$. Then 
one may
verify that the barycenter of the triple  $(p(x_0,v_0),p(x_1,v_1),p(x_2,v_2))\in V(T')^3$ 
is the vertex $a\Gamma_v$ where $x_0^{-1}x_1=a\gamma_1b$, $x_1^{-1}x_2=b^{-1}\gamma_2c$, 
and $\gamma_1,\gamma_2\in \G_v$ satisfy $\gamma_1\neq \gamma_2^{-1}$. 
Using this fact and Equality~\eqref{expression} it is easy to verify
that also
the square on the left is commutative.

Summarizing, as a corollary of Rolli's result we have shown that the image of 
$\mu\circ h^2\circ \overline d$ is a big subspace of coboundaries in $\oplus\linfa(S_v^3)^{\G_v}$ that are not taken by $\psi^2$ to 
coboundaries in $\linfa(S_{\mathcal{G}}^3)^{\G}$.
In particular, the restriction
of $\psi^2$ to bounded cocycles does not induce a well-defined map in bounded cohomology.
\end{rem}

\begin{rem}\label{countable:rem}
 The assumption that $G$ is finite did not play an important role in our proof of Theorem~\ref{thm:bounded:graph}. 
 Let us suppose that $G$ is countable, and take an element $\phi\in \hb^n(\Gamma)$. Then 
 the restriction $\h (i_v^n)(\phi)\in \hb^n(\G_v)$ can be non-null for infinitely many
$v\in V(G)$. However, we have $\|\h(i_v^n)(\phi)\|_\infty\leq \|\phi\|_\infty$ for every $v\in V(G)$, 
so there exists a well-defined map
$$
\prod_{v\in V(G)} \h (i^n_v) \ \colon \ \hb^n(\G)\longrightarrow \left(\prod_{v\in V(G)} \hb^n (\G_v)\right)^{\rm ub} \ ,
$$
where $\left(\prod_{v\in V(G)} \hb^n (\G_v)\right)^{\rm ub}$ is the subspace
of uniformly bounded elements of $\prod_{v\in V(G)} \hb^n (\G_v)$.
Our arguments easily extend to the case when $G$ is countable to prove
that, for every $n\geq 2$, there exists an isometric embedding
$$
\Theta^n\colon
\left(\prod_{v\in V(G)} \hb^n (\G_v)\right)^{\rm ub} \longrightarrow \hb^n(\G)
$$
which provides a right inverse to $\prod_{v\in V(G)} \h (i^n_v)$.
\end{rem}

\section{Mapping cones and Gromov Equivalence Theorem}\label{sec:equi}
Let $(X,Y)$ be a topological pair.
As mentioned in the introduction, Gromov considered in \cite{Gromov_82} 
the one-parameter family of norms on $\ch_n(X)$ defined by $\|c\|_1(\theta)=\|c\|_1+\theta\|\partial_n c\|_1$. 
All these norms are equivalent but distinct, and $\ch_n(Y)$ is a closed subspace 
of $\ch_n(X)$ with respect to any of these norms. Therefore, the norm $\|\cdot \|_1(\theta)$ descends to a quotient norm
on $\ch_n(X,Y)$, and to a quotient seminorm on $\h_n(X,Y)$.
All these (semi)norms will be denoted by $\|\cdot\|_1(\theta)$. They admit a useful description that exploits
a cone construction for relative singular homology analogous to
Park's cone construction for relative $\ell^1$-homology~\cite{Park} (see also \cite{Loeh}). 

Let us denote by
$i_n\colon \ch_n(Y)\to \ch_n(X)$ the map induced by the inclusion $i\colon Y\to X$.
The homology mapping cone complex of $(X,Y)$ 
is the complex 
$$\left(\ch_n(Y\to X),\overline{d}_n\right))=\left(\ch_n(X)\oplus \ch_{n-1}(Y), 
\left(\begin{smallmatrix}\partial_n&i_{n-1}\\0&-\partial_{n-1}\end{smallmatrix}\right)\right),$$ 
where $\partial_\bullet$ denotes the usual differential both of $\ch_\bullet(X)$ and of $\ch_\bullet (Y)$.
The homology of the mapping cone $(\ch_\bullet(Y\to X), \overline{d}_\bullet)$ 
is  denoted by $\h_\bullet(Y\rightarrow X)$. 
For every $n\in\mathbb{N}$, $\theta\in[0,\infty)$ one can endow $\ch_n(Y\to X)$ with the norm
$$\|(u,v)\|_1(\theta)=\|u\|_1+ \theta \|v\|_1\ , $$
which induces in turn a seminorm, still denoted by $\|\cdot \|_1 (\theta)$, 
on $\h_n(Y\rightarrow X)$.\footnote{In~\cite{Park_hom}, Park restricts her attention
only to the case $\theta\geq 1$.}

The chain map
\begin{equation}\label{betaPark:eq}
\beta_n\colon \ch_n(Y\to X)\longrightarrow \ch_n(X,Y)\, ,\qquad
\beta_n (u,v)=[u]
\end{equation}
induces a map $\h(\beta_n)$
in homology. 

\begin{lemma} \label{lem: theta iso} The map
$$\h(\beta_n)\colon\left( \h_n(Y\rightarrow X),\|\cdot \|_1(\theta)\right)  \longrightarrow \left( \h_n(X,Y),\|\cdot \|_1(\theta)\right)$$
is an isometric isomorphism for every $\theta\in [0,+\infty)$.
\end{lemma}

\begin{proof} 
It is immediate to check that $\h(\beta_n)$ admits
the inverse map 
$$
\h_n(X,Y)\to \h_n(Y\to X)\, ,\qquad 
[u]\mapsto [(u,-\partial_n u)]\, .
$$ 
Therefore, $\h(\beta_n)$ is an isomorphism, and we are left to show
that it is norm-preserving.

Let us set
$$
\beta'_n\colon \ch_n(Y\to X)\to \ch_n (X)\, ,\qquad 
\beta'_n(u,v)=u\, .
$$
By construction, $\beta_n$ is the composition of $\beta_n'$ with the 
natural projection $\ch_n(X)\to \ch_n(X,Y)$.
Observe that an element $(u,v)\in \ch_n(Y\to X)$ is a cycle if and only if
$\partial_nu=-i_{n-1}(v)$. As a consequence,
although the map $\beta'_n$ is not norm non-increasing in general, 
it does preserve norms when restricted to $\Zz\ch_n(Y\rightarrow X)$. 
Moreover, every chain in $C_n(X)$ representing a relative cycle is contained in
$\beta'_n (\Zz\ch_n(Y\rightarrow X))$,
and this concludes the proof.
\end{proof}

As is customary when dealing with seminorms in homology, in order to control the seminorm $\|\cdot \|_1(\theta)$ it is useful to 
study the topological dual of $(\ch_n(Y\to X),$ $\|\cdot\|_1(\theta))$, and exploit duality. 
If $(C_\bullet,d_\bullet)$ is a normed chain complex (i.e.~a chain complex of normed real vector spaces), 
then for every $n\in\mathbb{N}$ one may consider the topological
dual $D^n$ of $C_n$, endowed with the dual norm. 
The differential $d_n\colon C_n\to C_{n-1}$ induces a differential $d^{n-1}\colon D^{n-1}\to D^n$, 
and we say that $(D^\bullet,d^\bullet)$ is the dual normed chain complex of $(C_\bullet,d_\bullet)$. 
The homology (resp.~cohomology) of the complex $(C_\bullet,d_\bullet)$ (resp.~$(D^\bullet,d^\bullet)$) 
is denoted by $\h_\bullet(C_\bullet)$ (resp.~$\hb^\bullet(D^\bullet)$). We denote the norms on $C_n$ and $D^n$
and the induced seminorms on $\h_n(C_\bullet)$ and $\hb^n(D^\bullet)$ respectively
by $\|\cdot \|_C$ and $\|\cdot\|_D$.
The duality pairing between $D^n$ and $C_n$ induces the \emph{Kronecker product}
$$
\langle \cdot,\cdot \rangle\colon \hb^n(D^\bullet)\times \h_n(C_\bullet)\to \RR\ .
$$

By the Universal Coefficient Theorem, taking (co)homology commutes with taking \emph{algebraic} duals.
However, this is no more true when replacing algebraic duals with topological duals,
so $\hb^n(D^\bullet)$ is not isomorphic to
the topological dual of $\h_n(C_\bullet)$ in general (see e.g.~\cite{Loeh} for a thorough discussion of this issue). 
Nevertheless, the following well-known
consequence of Hahn-Banach Theorem establishes an important relation between
$\hb^n(D^\bullet)$ and $\h_n(C_\bullet)$. 
We provide a proof for the sake of completeness (and because in the available formulations of this result the maximum is replaced by a supremum).
\begin{lemma}\label{lemma:duality}
 Let $(C_\bullet,\|\cdot\|_C)$ be a normed chain complex with dual normed chain complex $(D^\bullet,\|\cdot\|_D)$. 
 Then, for every $\alpha\in \h_n(C_\bullet)$ we have 
$$
\|\alpha\|_C=\max \{ \langle\beta,\alpha\rangle\, |\,  \beta\in \hb^n(D^\bullet),\,  \|\beta\|_D\leq 1\}\ .
$$
\end{lemma}

\begin{proof}
The inequality $\geq$ is obvious.
Let $a\in C_n$ be a representative of $\alpha$. In order to conclude 
it is enough to find an element $b\in D^n$
such that $d^nb=0$, $b(a)=\|\alpha\|_C$ and $\| b\|_D\leq 1$.
If $\|\alpha\|_C=0$ we may take $b=0$. Otherwise,
let $V\subseteq C_n$ be the closure of 
$d_{n-1} C_{n-1}$ in $C_n$, and put on the quotient
$W:=C_n/V$ the induced seminorm $\|\cdot \|_W$. Since $V$ is closed,
such seminorm is in fact a norm.
By construction, $\|\alpha\|_C=\| [a]\|_W$. Therefore, Hahn-Banach Theorem provides
a functional $\overline{b}\colon W\to \mathbb{R}$ with operator norm one such that $\overline{b}([a])=\|\alpha\|_C$.
We obtain the desired element $b\in D^n$ by pre-composing $\overline{b}$  with the projection $C_n\to W$.
\end{proof}

Let us come back to  
the mapping cone for the homology of a pair $(X,Y)$.
For $\theta\in (0,\infty)$, the dual normed chain complex of 
$(\ch_n(Y\to X),\|\cdot \|_1(\theta))$
is
Park's mapping cone for relative bounded cohomology~\cite{Park}, that is the complex
$$
(\cb^n (Y\to X),\overline d^n)=\left(\cb^n(X)\oplus \cb^{n-1}(Y),\left(\begin{smallmatrix}d^n&0\\-i^n&-d^{n-1}\end{smallmatrix}\right)\right)
$$
endowed with the norm 
$$
\|(f,g)\|_\infty(\theta)=\max\{\|f\|_\infty ,\theta^{-1}
\|g\|_\infty\}. 
$$
We endow the cohomology $\hb^n(Y\rightarrow X)$
of the complex $(\cb^n(Y\to X),\overline{d}^n)$ with the quotient seminorm, which will be still denoted
by $\|\cdot\|_\infty(\theta)$.
The chain map
\bqn
\beta^n\colon \cb^n(X,Y) \longrightarrow \cb^n(Y\to X),\qquad
\eqn
induces an isomorphism between $\hb^n(X,Y)$ and $\hb^n(Y\to X)$
(see~\cite{Park}, or the first part of the proof of Proposition~\ref{fund:prop}). 
If we assume that the fundamental group of every component of $Y$ is amenable, then we can improve this result as follows:

\begin{prop}\label{fund:prop}
Suppose that the fundamental group of every component of $Y$ is amenable. Then,
for every $n\geq 2$, $\theta\in (0,\infty)$,
the map 
$$
\h(\beta^n)\colon \left(\hb^n(X,Y),\|\cdot\|_\infty\right)\to \left(\hb^n(Y\to X), \|\cdot \|_\infty(\theta)\right)
$$
is an isometric isomorphism.
\end{prop}
\begin{proof} 
Let us first prove that $\h(\beta^n)$ is an isomorphism (here we do not use any hypothesis
on $Y$).
To this aim, it is enough to show that the composition
\begin{equation}\label{comp:eq}
\xymatrix{
\Zz\ch^n_b(X,Y)\ar[r]^{\beta^n} & \Zz\ch^n_b(Y\to X) \ar[r]&  \hb^n(Y\to X)
}
\end{equation}
is surjective with kernel $d\cb^{n-1}(X,Y)$. For any $g\in \cb^\bullet(Y)$ 
we denote by $g'\in \cb^\bullet(X)$ the extension of $g$ 
that vanishes on simplices with image not contained in $Y$.
Let us take $(f,g)\in\Zz \ch^n_b(Y\to X)$. From $\overline{d}^n(f,g)=0$
we deduce that $f+dg'\in \Zz\ch^n_b(X,Y)$. Moreover, $(f+dg',0)-(f,g)=-\overline{d}^{n-1}(g',0)$, 
so the map~\eqref{comp:eq} above is surjective.
Finally, if $f\in \Zz \cb^n(X,Y)$ and $\overline{d}^{n-1}(\alpha,\beta)=(f,0)$, 
then $\alpha+d\beta'$ belongs to $\cb^{n-1}(X,Y)$ and $d(\alpha+d\beta')=f$. 
This concludes the proof that $\h(\beta^n)$ is an isomorphism.

Let us now suppose that the fundamental group of each component of $Y$ is amenable.
We
consider the
chain map $$\gamma^\bullet \colon \cb^\bullet (Y\to X)\to \cb^\bullet(X),\qquad(f,g)\mapsto f\ .$$
For every $n\in\mathbb{N}$
the composition $\gamma^n\circ \beta^n$ coincides with the inclusion
$j^n\colon \cb^n(X,Y)\to \cb^n(X)$.
By Theorem~\ref{thm:main}, for every $n\geq 2$ the map $\h(j^n)$ is an isometric isomorphism. 
Moreover, both $\h(\gamma^n)$ and $\h(\beta^n)$ are norm non-increasing, 
so we may conclude that the isomorphism $\h(\beta^n)$ is isometric for every $n\geq 2$.
\end{proof}

Putting together Proposition~\ref{fund:prop} and the main theorem of~\cite{Loeh}
we obtain the following result (which may be easily deduced also from 
Proposition~\ref{fund:prop} and Lemma~\ref{lemma:duality}):

\begin{cor}\label{fund:cor}
Suppose that the fundamental group of every component of $Y$ is amenable. Then,
for every $n\geq 2$, $\theta\in (0,\infty)$,
the map 
$$
\h(\beta_n)\colon \left(\h_n(Y\to X),\|\cdot \|_1(\theta)\right)\to  \left(\h_n(X,Y),\|\cdot\|_1\right)
$$
is an isometric isomorphism.
\end{cor}

We are now ready to conclude the proof of Gromov's Equivalence Theorem (Theorem~\ref{thm:equi} here). 
Under the assumption that the fundamental group
of every component of $Y$ is amenable, 
Lemma~\ref{lem: theta iso} and Corollary~\ref{fund:cor} imply that the identity
between $\left(\h_n(X,Y),\|\cdot\|_1\right)$ and $\left(\h_n(X,Y),\|\cdot\|_1(\theta)\right)$ is an isometry for every $\theta>0$.
The conclusion follows from the fact that, by definition,
$\|\cdot \|_1(0)=\|\cdot \|_1$ and $\|\cdot\|_1(\infty)=\lim_{\theta\to \infty} \|\cdot\|_1(\theta)$.

\section{Additivity of the simplicial volume}\label{sec:additivity}
Let us  recall that if $M$ is a compact connected orientable
$n$-manifold, the simplicial volume of $M$ is defined as
$$
\| M,\bb M\|=\| [M,\bb M]\|_1\ ,
$$
where $[M,\bb M]\in \ch_n(M,\bb M)$ is the image of the integral fundamental class
of $M$ via the change of coefficients homomorphism induced by the inclusion
$\mathbb{Z}\hookrightarrow \mathbb{R}$.

Let $G$ be a finite graph and let us associate to any vertex $v\in V(G)$ a compact oriented $n$-manifold $(M_v,\partial M_v)$ 
and to any edge $e\in E(G)$ a closed oriented $(n-1)$-manifold $S_e$ together 
with an orientation preserving homeomorphism $f_e\colon S_e\to \partial_e M_{t(e)}$, where $\partial_e M_{t(e)}$ is a connected component
of $\partial M_{t(e)}$. We also require that
$S_{\overline e}$ is equal to $S_e$ with reversed orientation, and that 
the images of $f_e$ and $f_{e'}$ are distinct whenever $e,e'$ are distinct
edges of $G$. 
We denote by $M$ the quotient of the union $\left(\bigcup_{v\in V(G)} M_v\right)\cup
\left(\bigcup_{e\in \bar E(G)} S_e\right)$ with respect to the identifications induced
by the maps $f_e$, $e\in E(G)$. Of course, $M$ is just the manifold obtained by gluing the $M_v$ along
the maps $f_e\circ f_{\overline{e}}^{-1}$, $e\in \bar{E}(G)$. We also assume that $M$ is connected.

For every $e\in E(G)$ we identify $S_e$ with the corresponding hypersurface in $M$,
and we
denote by $\Ss$ the union $\bigcup_{e\in \bar E(G)} {S}_e\subseteq M$. 
The inclusion $i_v\colon (M_v,\bb M_v)\to (M,\mathcal{S}\cup \bb M)$  is a map of pairs inducing a norm non-increasing map in cohomology 
$$
i_v^n\colon \hb^n(M,\mathcal{S}\cup \bb M
)\to \hb^n(M_v,\bb M_v) \ .$$
Moreover, since any component of $\bb M \cup \Ss$ has amenable fundamental group, 
and every compact manifold has the homotopy type of a finite CW-complex~\cite{RS}, 
we may compose the isomorphisms $\hb^n(M,\bb M)\cong\hb^n(M)$, $\hb^n(M)\cong\hb^n(M,\bb M\cup \Ss)$ provided by
Theorem \ref{thm:main}, thus getting an isometric isomorphism 
$$
\zeta^n\colon \hb^n(M,\bb M)\to \hb^n(M,\bb M\cup \Ss)\ .
$$
This map is the inverse of the map induced by the inclusion of pairs $(M,\bb M)\to(M,\bb M\cup \Ss) $.
Finally, we define the norm non-increasing map 
$$\zeta^n_v=i^n_v\circ\zeta^n\colon \hb^n(M,\bb M)\to \hb^n(M_v,\bb M_v)\ .$$

\begin{lemma}\label{somma}
For every $\varphi\in \hb^n(M,\bb M)$ we have
$$
\langle \varphi,[M,\bb M]\rangle=\sum_{v\in V(G)} \langle \zeta^n_v(\varphi), [M_v,\bb M_v]\rangle\ .
$$
\end{lemma}
\begin{proof}
Let $c_v\in \ch_n(M_v)$ be a real chain representing the fundamental class of $M_v$.
With an abuse, we identify any chain in $M_v$ with the corresponding chain in
$M$, and we 
set
$c=\sum_{v\in V(G)} c_v \in \ch_n(M)$.
We now suitably modify $c$ in order to obtain a relative fundamental cycle for
$M$. It is readily seen that $\bb c_v$ is the sum of real fundamental cycles of the boundary components
of $M_v$. Therefore, since the gluing maps defining $M$ are
orientation-reversing, 
we may choose a chain $c'\in \oplus_{e\in \bar{E}(G)} \ch_{n}({S}_e)$ such that $\bb c-\bb c'\in \ch_{n-1}(\bb M)$.
We  set $c''=c-c'$. By construction $c''$ is a relative cycle
in $\ch_n (M,\bb M)$, and it is immediate to check 
that it is in fact a relative fundamental cycle for $M$.
Let now $\psi\in \cb^n(M,\mathcal{S}\cup \bb M)$ be a representative of 
$\zeta^n(\varphi)$.
By definition we have
$$
\psi(c)=\sum \psi(c_v)=\sum \langle \zeta^n_v(\varphi), [M_v,\bb M_v]\rangle \ .
$$
On the other hand, since $\psi$ vanishes on chains supported on $\mathcal{S}$, we also have
$$
\psi(c)=\psi(c''+c')=\psi(c'')=\langle \varphi,[M,\bb M]\rangle\ ,
$$
and this concludes the proof.
\end{proof}

Let us now proceed with the proof of Theorem~\ref{simpl:thm}.
In order to match the notation with the statement of Theorem~\ref{simpl:thm}, 
we henceforth denote by $\{1,\ldots,k\}$ the set of vertices of $G$. 
By Lemma~\ref{lemma:duality}
we may choose an element $\varphi\in\hb^n(M,\bb M)$
such that 
$$
\|M,\partial M\|=\langle  \varphi,[M,\bb M]\rangle \, ,\qquad 
\|\varphi\|_\infty\leq 1\ .
$$
Observe that $\|\zeta^n_v(\varphi)\|_\infty\leq \|\varphi\|_\infty\leq 1$
for every $v\in V(G)$, so by Lemma~\ref{somma}
$$
\|M,\partial M\| 
=
\langle \varphi,[M,\bb M]\rangle  =\sum_{v=1}^k \langle \zeta^n_v(\varphi),[M_v,\bb M_v]\rangle  \leq 
{\sum_{v=1}^k \|M_v,\bb M_v\|}\ .
$$
This proves the first part of Theorem~\ref{simpl:thm}.

\begin{rem}\label{final:rem}
The inequality
$$
\| M,\bb M\|\leq \| M_1,\bb M_1\|+\ldots+\|M_k,\bb M_k\|
$$
may also be proved by showing that ``small'' fundamental cycles for $M_1,\ldots,M_k$ may be glued together
to construct a ``small'' fundamental cycle for $M$.
The fundamental group of every ${S}_e\subseteq M$ is amenable, so 
the real singular chain module
$\oplus_{e\in \bar{E}(G)} \ch_{n-1}({S}_e)$
satisfies Matsumoto-Morita's \emph{uniform boundary condition} \cite{Matsumoto_Morita}.
This means that there exists 
$U>0$ such that every boundary $\bb c\in \ch_{n-1}({S}_e)$ satisfies the equation $\bb c=\bb c'$ for some
$c'\in \ch_n({S}_e)$ such that $\|c'\|_1\leq U\cdot \|\bb c\|_1$. 

Let now $\vare>0$ be given.
By Corollary~\ref{Thurston's version},
for every $v$ we may choose a real relative fundamental cycle $c_v\in \ch_n(M_v,\bb M_v)$
such that 
$$\|c_v\|_1\leq \|M_v,\bb M_v\|+\vare\ ,\qquad \|\bb c_v\|_1\leq \vare\ .$$ 
Let us  set $c=c_1+\ldots+c_k\in \ch_n(M)$ (as above, we identify any chain in $M_v$  with
its image in $M$). 
As we did in
 Lemma~\ref{somma}, we may obtain a relative fundamental cycle $c''$
for $M$ by setting $c''=c-c'$, where $c'$ is a suitable chain in
$\oplus_{e\in \bar{E}(G)} \ch_{n}({S}_e)$. Moreover, by the uniform boundary condition we may choose $c'$ in such a way that 
$\|c'\|_1\leq Uk\vare$. 
Therefore, we eventually have
$$
\| M,\bb M\|  \leq \|c''\|_1\leq \|c\|_1+\|c'\|_1\leq \sum\|M_v,\bb M_v\| + k\vare+Uk\vare
$$
and this concludes the proof.
\end{rem}

To conclude the proof of Theorem~\ref{simpl:thm} we now consider
the case when $M$ is obtained via compatible gluings.
Therefore, if $K_e$ is the kernel of the map induced by $f_e$ on fundamental groups,
then  $K_e=K_{\overline{e}}$ for every $e\in E(G)$
(recall that $S_e=S_{\overline{e}}$, so both $K_e$ and $K_{\overline{e}}$
are subgroups of $\pi_1(S_e)=\pi_1(S_{\overline{e}})$).
If we consider the graph of groups $\Gg$ with vertex groups $G_v=\pi_1(M_v)$ and 
edge groups $G_e=\pi_1(S_e)\slash K_e$, then van Kampen Theorem implies that 
$\pi_1(M)$ is the fundamental group of the graph of groups $\Gg$ (see \cite{Wall} for full details). 

\begin{prop}\label{comp:prop}
For every
$(\varphi_1,\ldots,\varphi_k)\in \oplus_{v=1}^k \hb^n (M_v,\bb M_v)$, there exists 
$\varphi\in\hb^n(M, \bb M)$ such that
$$
\|\varphi\|_\infty\leq \|(\varphi_1,\ldots,\varphi_k)\|_\infty\,,
\quad \zeta^n_v(\varphi)=\varphi_v\, ,\ v=1,\ldots,k\ .
$$
\end{prop} 
\begin{proof}
 The proposition follows at once from Theorem~\ref{thm:baricenter} and
 the commutativity of 
the following diagram:
$$
\xymatrix{
\hb^n(M,\bb M) \ar[r] \ar[d]_{\oplus \zeta^n_v} & \hb^n(M)\ar[r]\ar[d]& \hb^n(\pi_1(M))\ar[d]^{\oplus i_v^n}\\
\oplus \hb^n(M_v,\bb M_v)\ar[r] & \oplus\hb^n(M_v)\ar[r]&\oplus \hb^n(\pi_1(M_v))\ ,
}
$$
where the horizontal arrows are, respectively, the isometric isomorphisms constructed in Theorem \ref{thm:main} and Ivanov's maps,
and the vertical arrows are given by restrictions.
\end{proof}

By Lemma~\ref{lemma:duality},
for every $v=1,\ldots,k$, 
we may choose an element $\varphi_v\in\hb^n(M_v,\bb M_v)$
such that 
$$
\|M_v,\partial M_v\|= \langle \varphi_v,[M_v,\bb M_v]\rangle \, ,\qquad 
\|\varphi_v\|_\infty\leq 1\ ,
$$
and Proposition~\ref{comp:prop} implies that
there exists $\varphi\in \hb^n(M,\bb M)$ such that
$$
\|\varphi\|_\infty\leq 1\, ,\quad \zeta^n_v(\varphi)=\varphi_v\, ,\ v=1,\ldots,k\ .
$$
Using Lemma~\ref{somma} we get
$$
 \sum_{v\in V(G)} \|M_v,\bb M_v\| =
 \sum_{v\in V(G)} \langle \varphi_v,[M_v,\bb M_v]\rangle=\langle \varphi, [M,\bb M]\rangle
\leq \|M,\bb M\|\ ,
$$
which finishes the proof of Theorem~\ref{simpl:thm}.

\begin{rem}\label{counter}
The following examples show that 
the hypotheses of Theorem~\ref{simpl:thm} should not be too far from being the weakest
possible.

Let $M$ be a hyperbolic $3$-manifold with connected geodesic boundary. It is well-known
that $\partial M$ is 
$\pi_1$-injective in $M$. We fix a pseudo-Anosov homeomorphism $f\colon \partial M\to\partial M$, and for every $m\in \mathbb{N}$
we denote by $D_m M$ the twisted double obtained
by gluing two copies of $M$ along the homeomorphism $f^m\colon \partial M\to \partial M$
(so $D_0 M$ is the usual double of $M$).
It is shown in~\cite{Jungreis} that 
$$\| D_0 M\|<2\cdot \|M,\partial M \|\ .$$ 
On the other hand, 
by~\cite{Soma2} we have  $\lim_{m\to \infty} {\rm Vol}\, D_{m} M=\infty$.
But ${\rm Vol}\, N=v_3 \cdot \|N\|$ for every closed hyperbolic $3$-manifold $N$,
where $v_3$ is a universal constant~\cite{Gromov_82, Thurston_notes}, so 
$\lim_{m\to \infty} \| D_{m} M \|=\infty$, and the inequality $$\|D_{m} M\|>2\cdot \|M,\partial M\|$$
holds for infinitely many $m\in\mathbb{N}$.
This shows that, 
even in the case when each $S_e$ is $\pi_1$-injective in $M_{t(e)}$, 
no inequality between $\|M,\bb M\|$ and $\sum_{v=1}^k \|M_v,\bb M_v\|$ holds
in general
if one drops the requirement that 
the fundamental group of every $S_e$ is amenable.

On the other hand, if $M_1$ is (the natural compactification of) the once-punctured torus and $M_2$ is the $2$-dimensional disk,
then the manifold $M$ obtained by gluing $M_1$ with $M_2$ along $\bb M_1\cong \bb M_2\cong S^1$ is a torus, 
so
$$
\|M\|=0<2+0=\| M_1,\bb M_1\|+ \|M_2,\bb M_2\|\ .
$$
This shows that, even 
in the case when the fundamental group of every $S_e$ is amenable, 
the equality $\|M,\bb M\|=\sum_{j=1}^k \|M_j,\bb M_j\|$
does not hold in general if one drops the requirement that the gluings
are compatible.
\end{rem}

\vskip1cm

\providecommand{\bysame}{\leavevmode\hbox to3em{\hrulefill}\thinspace}
\providecommand{\MR}{\relax\ifhmode\unskip\space\fi MR }
\providecommand{\MRhref}[2]{%
  \href{http://www.ams.org/mathscinet-getitem?mr=#1}{#2}
}
\providecommand{\href}[2]{#2}

\end{document}